\documentclass [10pt,a4paper,reqno]{amsart}
\usepackage{tikz-cd}
\usepackage{centernot}
\usepackage{xypic}
\usepackage{url}

\xyoption{all}

\newif\iffurther
\furthertrue

\newtheorem{thm}{Theorem}[section] 

\newtheorem{cor}[thm]{Corollary}

\newtheorem{lem}[thm]{Lemma}
\newtheorem{prop}[thm]{Proposition}

\newtheorem{rem}[thm]{Remark}

\newtheorem*{thma1}{\textbf{Theorem A1}}
\newtheorem*{thma2}{\textbf{Theorem A2}}
\newtheorem*{thmb1}{\textbf{Theorem B1}}
\newtheorem*{thmb2}{\textbf{Theorem B2}}
\newtheorem*{thmc}{\textbf{Theorem C}}
\newtheorem*{thmd}{\textbf{Theorem D}}

\def\[{\left[}
\def\]{\right]}

\def\S{\mathcal{S}}

\def\GK{{\operatorname{GKdim}}}

\def\Lin{{\operatorname{Lin}}}

\def\Span{{\operatorname{Span}}}

\def\charac{{\operatorname{char}}}

\def\coeff{{\operatorname{coeff}}}
\def\GKdim{{GK-dimension}}

\long\def\forget#1\forgotten{{}}

\begin{document}

\title[Algebras of intermediate growth]{Nil algebras, Lie algebras and wreath products with intermediate and oscillating growth}

\author{Be'eri Greenfeld}
\address{Department of Mathematics, University of California, San Diego, La Jolla, CA, 92093, USA}
\email{bgreenfeld@ucsd.edu}

\author{Efim Zelmanov}
\address{Department of Mathematics, University of California, San Diego, La Jolla, CA, 92093, USA}
\email{efim.zelmanov@gmail.com}


\begin{abstract}
We construct finitely generated nil algebras with prescribed growth rate. In particular, any increasing submultiplicative function is realized as the growth function of a nil algebra up to a polynomial error term and an arbitrarily slow distortion.

We then move on to examples of nil algebras and domains with strongly oscillating growth functions and construct primitive algebras for which the Gel'fand-Kirillov dimension is strictly sub-additive with respect to tensor products, thus answering a question from \cite{KrauseLenagan,KrempaOkninski}.




\end{abstract}

\maketitle

\section{Introduction}




Let $F$ be a field and let $A$ be a finitely generated, infinite-dimensional $F$-algebra.
Fixing a finite-dimensional generating subspace $A=F\left<V\right>$, the growth of $A$ with respect to $V$ is defined to be the function:
$$g_{A,V}(n)=\dim_F \left(F+V+V^2+\cdots+V^n\right)$$
If $1\in V$ then equivalently $g_{A,V}(n)=\dim_F V^n$. This function obviously depends on the choice of $V$, but only up to the following equivalence relation.
We say that $f\preceq g$ if $f(n)\leq Cg(Dn)$ for some $C,D>0$ and for all $n\in \mathbb{N}$, and $f\sim g$ (asymptotically equivalent) if $f\preceq g$ and $g\preceq f$. Therefore when talking about `the growth of an algebra' one refers to $g_A(n)$ as the equivalence class of the functions $g_{A,V}(n)$ (for an arbitrary $V$) under the equivalence relation $\sim$.
A general reference for growth of algebras is \cite{KrauseLenagan}.

One of the most fundamental problems in combinatorial algebra is to characterize the possible growth rates of groups, algebras and Lie algebras.
The first example of a group of intermediate growth (that is, super-polynomial but subexponential) was given by Grigorchuk \cite{GrigorchukIntermediate2}.

How do growth functions of algebras look like? 
Let $g_A(n)$ be the growth function of an infinite-dimensional, finitely generated algebra $A$ with respect to a fixed generating subspace. Then $g_A(n)$ is increasing ($g_A(n)<g_A(n+1)$) and submultiplicative ($g_A(n+m)\leq g_A(n)g_A(m)$). Bell and Zelmanov \cite{BellZelmanov} found a concrete characterization of growth functions of algebras, from which it follows that any increasing and submultiplicative function is equivalent to the growth of some finitely generated algebra up to a linear error term, which is the best possible approximation.

Of particular interest is the class of associative nil algebras. Recall that an associative algebra is called nil if all of its elements are nilpotent. For many years it has been an open problem if an infinite-dimensional nil algebra can have a polynomially bounded growth. In 2007 Lenagan and Smoktunowicz \cite{LenaganSmoktunowicz} (see also \cite{LenaganSmoktunowiczYoung}) constructed infinite-dimensional nil algebras of polynomially bounded growth over countable fields. Bell and Young \cite{BellYoung} constructed infinite-dimensional nil algebras over arbitrary fields whose growth is bounded above by an arbitrarily slow super-polynomial function. In \cite{AAJZ_ERA} it is conjectured that any (non-linear) function which occurs as the growth rate of an algebra is realizable as the growth of a nil algebra; this can be thought of as a strong quantitative version of the Kurosh Problem.

We prove the following approximation of the aforementioned conjecture, establishing its validity up to a polynomial error term and an arbitrarily small distortion.

\begin{thma1}[{Arbitrary growth: Countable fields}]
Let $f\colon \mathbb{N}\rightarrow \mathbb{N}$ be an increasing, submultiplicative function and let $\delta\colon \mathbb{N}\rightarrow \mathbb{N}$ be an arbitrarily slow, non-decreasing function tending to infinity. 

Let $F$ be a countable field. Then there exists a finitely generated nil algebra $R$ over $F$ whose growth function $g_R(n)$ satisfies:
$$ f\left(\frac{n}{\delta(n)}\right) \preceq g_R(n) \preceq p(n)\cdot f(n) $$
where $p(n)$ is a polynomial which is independent of $f$.
\end{thma1}

The situation over uncountable fields is different and requires a mild relaxation of the formulation:

\begin{thma2}[{Arbitrary growth: Arbitrary fields}]
Let $f\colon \mathbb{N}\rightarrow \mathbb{N}$ be an increasing, submultiplicative function and let $\omega\colon \mathbb{N}\rightarrow \mathbb{N}$ be an arbitrarily slow, non-decreasing super-polynomial function.
Let $\delta\colon \mathbb{N}\rightarrow~\mathbb{N}$ be an arbitrarily slow, non-decreasing function tending to infinity.

Let $F$ be an arbitrary field. Then there exists a finitely generated nil algebra $R$ over $F$ whose growth function $g_R(n)$ satisfies:
$$ f\left(\frac{n}{\delta(n)}\right) \preceq g_R(n) \preceq \omega(n)\cdot f(n). $$
\end{thma2}

Kassabov and Pak \cite{KassabovPak} constructed groups whose growth functions oscillate between given intermediate functions (e.g.~$\exp(n^{4/5})$) and an arbitrarily rapid subexponential function; see also \cite{BartholdiErschler}. Bartholdi and Erschler \cite{Permutational} constructed periodic groups with explicitly given growth functions, ranging over a wide variety of intermediate growth functions.

\begin{thmb1}[{Oscillating growth: Countable fields}]
Let $f\colon\mathbb{N}\rightarrow\mathbb{N}$ be an arbitrarily rapid subexponential function.
Let $F$ be a countable field. Then there exists a finitely generated nil algebra $R$ over $F$ such that:
\begin{itemize}
\item $g_R(n)\leq n^{6+\varepsilon}$ infinitely often for every $\varepsilon>0$; and
\item $g_R(n)\geq f(n) $ infinitely often.
\end{itemize}
\end{thmb1}

\begin{thmb2}[{Oscillating growth: Arbitrary fields}]
Let $f\colon\mathbb{N}\rightarrow\mathbb{N}$ be an arbitrarily rapid subexponential function and let $\omega\colon \mathbb{N}\rightarrow \mathbb{N}$ be an arbitrarily slow, non-decreasing, super-polynomial function.

Let $F$ be an arbitrary field. Then there exists a finitely generated nil algebra $R$ over $F$ such that:
\begin{itemize}
\item $g_R(n)\leq \omega(n) $ infinitely often; and
\item $g_R(n)\geq f(n) $ infinitely often.
\end{itemize}
\end{thmb2}

These are the first examples of nil algebras whose growth oscillates between a polynomial and an arbitrarily rapid subexponential function. Monomial algebras with oscillating growth were constructed in \cite{BBL}.

Petrogradsky \cite{PetrogradskyLie} constructed a far-reaching Lie-theoretic analogy of these phenomena. Namely, he constructed nil (restricted) Lie algebras over fields of positive characteristic, whose growth functions oscillate between a function very close to linear and a function very close (though not arbitrarily close) to exponential. The growth of a finitely generated Lie algebra $L$ with respect to a generating subspace $V$ is the dimension of the space spanned by Lie monomials of length at most $n$ in the elements of $V$, that is, $g_L(n)=\dim_F(V+V^{[2]}+\cdots+V^{[n]})$.

Using Theorems B1, B2 we can construct nil Lie algebras with oscillating growth functions over fields of arbitrary characteristic. A Lie algebra $L$ is nil if $ad_x$ is nilpotent for every $x\in L$.

\begin{cor} \label{oscillatingLie}
Let $f\colon\mathbb{N}\rightarrow\mathbb{N}$ be an arbitrarily rapid subexponential function.
Let $F$ be a countable field of $\charac(F)\neq 2$. Then there exists a finitely generated nil Lie algebra $L$ over $F$ such that:
\begin{itemize}
\item $g_L(n)\leq n^{6+\varepsilon}$ infinitely often for every $\varepsilon>0$; and
\item $g_L(n)\geq f(n) $ infinitely often.
\end{itemize}
\end{cor}

\begin{cor} \label{oscillatingLie2}
Let $f\colon\mathbb{N}\rightarrow\mathbb{N}$ be an arbitrarily rapid subexponential function and let $\omega\colon \mathbb{N}\rightarrow \mathbb{N}$ be an arbitrarily slow, non-decreasing, super-polynomial function.

Let $F$ be an uncountable field of $\charac(F)\neq 2$. Then there exists a finitely generated nil Lie algebra $L$ over $F$ such that:
\begin{itemize}
\item $g_L(n)\leq \omega(n) $ infinitely often; and
\item $g_L(n)\geq f(n) $ infinitely often.
\end{itemize}
\end{cor}

Our next goal is to construct domains with oscillating intermediate growth.
There are some similarities between growth rates of domains and growth rates of groups. For instance, there are no known domains of super-polynomial growth slower than $\exp(\sqrt{n})$; any domain of subexponential growth admits a division ring of fractions, and it is a widely open problem whether the group algebra of a torsion-free group is always a domain.
As mentioned before, Kassabov and Pak \cite{KassabovPak} constructed finitely generated groups whose growth oscillates between $\exp(n^{4/5})$ and an arbitrarily rapid (subexponential) function. We prove:

\begin{thmc}
Let $f\colon \mathbb{N}\rightarrow \mathbb{N}$ be a subexponential function. Then there exists a finitely generated domain $A$ such that $g_A(n)\geq f(n)$ infinitely often, and $g_A(n)\leq~\exp\left(n^{\frac{3}{4}+\varepsilon}\right)$ infinitely often (for arbitrary $\varepsilon>0$).
\end{thmc}

Finally, we turn to a somewhat different construction of algebras whose growth oscillates within a prescribed \textit{polynomial} interval. Recall that the Gel'fand-Kirillov (GK) dimension of a finitely generated algebra $A=F\left<V\right>$ is:
$$ \GK(A) = \limsup_{n\rightarrow\infty} \frac{\log \dim_F V^n}{\log n} $$
namely the (optimal) degree of polynomial growth of $A$.
For algebras with `sufficiently regular' growth functions, the GK-dimension is additive with respect to tensor products. 
Warfield \cite{Warfield} proved that $\GK(A\otimes_F B)$ might be smaller than $\GK(A)+\GK(B)$, and gave an example where: $$ \GK(A\otimes_F B)=\max\{\GK(A),\GK(B)\}+2 $$ which is the minimum possible value for $\GK(A\otimes_F B)$. Let $\alpha=\GK(A),\beta=\GK(B)$ and suppose that $\alpha\leq \beta$. Warfield asked whether any value in the interval $[2+\beta,\alpha+\beta]$ is attainable for $\GK(A\otimes_F B)$. This was settled in the affirmative by Krempa and Okni\'nski \cite{KrempaOkninski}, even in a slightly wider generality. Both Warfield's and Krempa-Okni\'nski's constructions have large prime radicals, which is a common phenomenon among generic constructions of monomial algebras. 
This led Krempa and Okni\'nski to ask whether semiprime examples of this flavor can be found. The problem of finding Warfield-type examples among semiprime rings is mentioned also by Krause and Lenagan \cite[Page~167]{KrauseLenagan}.

We settle this problem in the affirmative. Our constructions have a symbolic dynamical origin, and arise as monomial algebras associated with Toeplitz subshifts with oscillating complexity growth.

\begin{thmd} \label{ThmD}
For arbitrary $2\leq \gamma \leq \alpha\leq \beta < \infty$ there exist finitely generated primitive monomial algebras $A,B$ such that: $$ \GK(A)=\alpha,\ \ \GK(B)=\beta,\ \ \GK(A\otimes_F B)=\beta+\gamma. $$
\end{thmd}

\textit{Conventions}. Given a function $f\colon \mathbb{N}\rightarrow \mathbb{N}$, we let $f(x)=f(\lfloor x \rfloor)$ for any real $x\geq 1$. Lie algebras are considered over fields of $\charac\neq 2$.

\section{Matrix wreath products}

We recall the construction of matrix wreath products from \cite{AAJZ_TransAMS}.
Fix an arbitrary base field $F$. Let $B$ be a finitely generated $F$-algebra and $A$ a countably generated $F$-algebra. Fix a finite-dimensional generating subspace for $B$, say, $V$. If $B$ is non-unital, we let $B'$ be its unital hull ($B'=F+B$).  

Let $A\wr B=B+\Lin_F(B',B'\otimes_F A)$, where $\Lin_F(*,*)$ is the space of $F$-linear maps. Then $A\wr B$ is endowed with a multiplicative structure, making it an $F$-algebra. For $f,g\in \Lin_F(B',B'\otimes_F A)$ let: 

$$ fg = (1\otimes \mu)(f\otimes 1)g $$
where $\mu\colon A\otimes_F A\rightarrow A$ is the multiplication map.
Multiplication on $B$ is given by the ring structure of $B$, so it remains to define a $B$-bimodule structure on $\Lin_F(B',B'\otimes_F A)$ which makes $A\wr B$ an associative ring. Given $f\in \Lin_F(B',B'\otimes_F A)$, we let $(fb)(x)=f(bx)$ and $(bf)(x)=(b\otimes 1)f(x)$.

With any linear map $\gamma \colon B'\rightarrow A$, associated is an element $$c_\gamma\in \Lin_F(B',B'\otimes_F A),$$ given by: $$c_\gamma (x)=1\otimes \gamma(x).$$ 
Finally, let $C=F\left<V,c_\gamma\right>\subseteq A\wr B$.

\begin{lem} \label{decomposition}
For any $b_1,\dots,b_s\in B$, we have: $$ c_\gamma b_1 c_\gamma b_2 \cdots c_\gamma b_s(x) = 1\otimes \gamma(b_1)\cdots \gamma(b_{s-1})\gamma(b_sx) $$
\end{lem}
\begin{proof}
Suppose that $\sigma,\tau\colon B'\rightarrow A$ are maps. Then: $$c_\tau c_\sigma(x)=1\otimes \tau(1)\sigma(x).$$ Indeed,
\begin{eqnarray*}
c_\tau c_\sigma (x) & = & (1\otimes \mu)(c_\tau \otimes 1)c_\sigma(x) \\
& = & (1\otimes \mu)(c_\tau \otimes 1)(1\otimes \sigma(x)) \\
& = & (1\otimes \mu)(c_\tau(1)\otimes \sigma(x)) \\
& = & (1\otimes \mu)(1\otimes \tau(1)\otimes \sigma(x)) \\
& = & 1\otimes \tau(1)\sigma(x).
\end{eqnarray*}
Therefore, given $s$ maps: $$c_{\tau_1}c_{\tau_2}\cdots c_{\tau_s}(x)=1\otimes \tau_1(1)\cdots \tau_{s-1}(1)\tau_s(x)$$
Since $c_\gamma b=c_{\gamma'}$, where $\gamma'(x)=\gamma(bx)$, we get that: 
$$ c_\gamma b_1 c_\gamma b_2 \cdots c_\gamma b_s(x) = 1\otimes \gamma(b_1)\cdots \gamma(b_{s-1})\gamma(b_sx) $$
as claimed.
\end{proof}

Let $g_B,g_C$ be the growth functions of $B,C$ with repsect to $V,V+Fc_\gamma$, respectively.
The growth of $C$ turns out to be closely related to the following measurement of growth of $\gamma$. Let:

$$ W_n = \sum_{i_1+\cdots+i_s\leq n} \gamma\left(V^{i_1}\right)\cdots \gamma\left(V^{i_s}\right) $$
and let $w_\gamma(n)=\dim_F W_n$. Obviously, $w_\gamma(n)$ is monotonely non-decreasing.

In \cite{AAJZ_TransAMS} it was proved that $g_C(n)\preceq g_B(n)^2w_\gamma(n)$, and if $\gamma$ satisfies an additional condition, called \textit{density}, then $g_C(n) \sim g_B(n)^2w_\gamma(n)$.
More concretely, $\gamma$ is dense if for any system of linearly independent elements $b_1,\dots,b_k\in B$ and any $0\neq a\in A$, there exists $b\in B$ such that $\gamma(b_ib)=0$ for all $1\leq i\leq k-1$ but $a\gamma(b_kb)\neq 0$.
Without density, we have a weaker lower bound, which suffices for our purposes.




\begin{lem} \label{growth_wreath}
Let $\gamma\colon B'\rightarrow A$ be a linear map. Then $w_\gamma(n)\leq g_C(2n)$ and $g_C(n)\leq g_B(n)^2w_\gamma(n)+g_B(n)$, and thus $w_\gamma(n) \preceq g_C(n) \preceq g_B(n)^2 w_\gamma(n)$.
\end{lem}
\begin{proof}
The inequality $g_C(n)\preceq g_B(n)^2w_\gamma(n)$ is \cite[Corollary~3.6]{AAJZ_TransAMS} and its quantitative form readily follows from \cite[Lemma~3.5]{AAJZ_TransAMS}.

For the other inequality, recall that by Lemma \ref{decomposition}, given $b_1,\dots,b_s\in~B$, the map $c_\gamma b_1 c_\gamma b_2 \cdots c_\gamma b_s$ takes the form $x\mapsto 1\otimes \gamma(b_1)\cdots \gamma(b_s x)$. In particular, the substitution map:
$$ T\colon \Lin_F(B',B'\otimes_F A) \rightarrow B'\otimes_F A $$
$$ T\colon \psi \mapsto \psi(1) $$
surjectively carries: $$T\left( \sum_{i_1+\cdots+i_s\leq n} c_\gamma V^{i_1} c_\gamma V^{i_2}\cdots c_\gamma V^{i_s}\right) = 1\otimes W_n$$
hence: $$\dim_F \sum_{i_1+\cdots+i_s\leq n} c_\gamma V^{i_1} c_\gamma V^{i_2}\cdots c_\gamma V^{i_s} \geq w_\gamma(n).$$ Since: $$\sum_{i_1+\cdots+i_s\leq n} c_\gamma V^{i_1} c_\gamma V^{i_2}\cdots c_\gamma V^{i_s} \subseteq \left(V+Fc_\gamma\right)^{\leq 2n},$$ it follows that $w_\gamma(n)\leq g_C(2n)$.
\end{proof}

\section{Linear maps with oscillating growth}

A (non-decreasing) function $f\colon \mathbb{N}\rightarrow\mathbb{N}$ is \textit{subexponential}, if $\limsup_{n\rightarrow \infty} \sqrt[n]{f(n)} =~ 1$.

\begin{prop} \label{Oscillating_transformation}
Let $B$ be a finitely generated, infinite-dimensional unital algebra. Let $f_1,f_2\colon\mathbb{N}\rightarrow\mathbb{N}$ be monotone non-decreasing functions such that: 
\begin{itemize}
\item $f_1(n)\xrightarrow{n\rightarrow\infty} \infty$ (perhaps very slowly); and
\item $f_2(n)$ is subexponential.
\end{itemize}
Then there exists a locally nilpotent algebra $A$ with a linear map $\gamma\colon B\rightarrow~A$ such that its growth function $w_\gamma$ satisfies:
\begin{itemize}
\item $w_\gamma(n)\leq f_1(n)$ infinitely often; and
\item $w_\gamma(n)\geq f_2(n)$ infinitely often.
\end{itemize}
\end{prop}

\begin{proof}
Assume that: $$1<d_1\leq d_2\leq \cdots$$
$$1<n_1<m_1<n_2<m_2<\cdots$$
are sequences of positive integers which are to be determined in the sequel.
In the countably generated free (non-unital) algebra:
$$ F\left<x_1,x_2,\dots\right>_{\geq 1} $$ consider the ideals:
$$ I_M=\left< x_{i_1}\cdots x_{i_{d_M}}\ |\ i_1,\dots, i_{d_M}\leq M \right> $$
for $M=1,2,\dots$ and notice that $I_{M'}\cap F\left<x_1,\dots,x_M\right>_{\geq 1} \subseteq I_M$ for $M'\geq M$, since $\{d_i\}_{i=1}^{\infty}$ is non-decreasing. Let $I=\sum_{M=1}^\infty I_M\triangleleft F\left<x_1,x_2,\dots\right>_{\geq 1}$. Let:
$$ A=F\left<x_1,x_2,\dots\right>_{\geq 1}/I $$
It is clear that $A$ is locally nilpotent.

Let $1\in V$ be a generating subspace of $B$. We define a linear map $\gamma\colon B\rightarrow~A$ along with coherent constraints on the sequences $\{d_i,m_i,n_i\}_{i=1}^{\infty}$ mentioned above.

Take $p\in \mathbb{N}$ such that $f_2(2p) \leq 2^p$ and let $n_1=2p$ (this is possible since $f_2$ is subexponential). Pick a basis $b_1,\dots,b_t$ for $V^{2p}=V^{n_1}$ such that $b_1,b_2\in V^2$. Set $\gamma(b_1)=x_1,\gamma(b_2)=x_2$ and $\gamma(b_i)=0$ for $2<i\leq t$. Let $d_1=\cdots=d_{n_1}=p+1$. Then all length-$p$ free monomials in $ x_1,x_2 \in \gamma(V^2) $ are linearly independent in $A$, so: 
\begin{eqnarray*}
f_2(n_1) & \leq & 2^p \\ & \leq & \dim_F \gamma\left(V^2\right)^p \\ & \leq & \dim_F W_{2p} \\ & = & w_\gamma(2p) \\ & = & w_\gamma(n_1)
\end{eqnarray*}

Suppose that $n_1,\dots, n_i,d_1,\dots,d_{n_i}$ were defined, and $\gamma$ was set on $V^{n_i}$ such that $\gamma\left(V^{n_i}\right)\subseteq \overline{F\left<x_1,\dots,x_{n_i}\right>}_{\geq 1}$ (namely, the subalgebra of $A$ generated by $x_1,\dots,x_{n_i}$). The latter is a nilpotent algebra of nilpotency index at most $d_{n_i}$, hence its dimension is at most $K=K(i,n_i,d_{n_i})$. Take $m_i\gg n_i$ such that $f_1(m_i)\geq K$. Fix a vector space complement $V^{m_i}=V^{n_i}\oplus W$ and set $\gamma\left( W\right)=0$. Then: 
\begin{eqnarray*}
w_\gamma(m_i) & = & \dim_F \sum_{j_1+\cdots+j_s\leq m_i} \gamma\left(V^{j_1}\right)\cdots \gamma\left(V^{j_s}\right) \\ 
& \leq & \dim_F F\left<\gamma\left(V^{n_i}\right)\right> \\
& \leq & \dim_F \overline{F\left<x_1,\dots,x_{n_i}\right>}_{\geq 1} \\
& \leq & K \leq f_1(m_i).
\end{eqnarray*}
Take also $ d_{n_i+1} = \cdots = d_{m_i} $ to be $d_{n_i}$.

Now suppose that $m_i$ has been defined and $\gamma$ was set on $V^{m_i}$ as above. Let us specify $n_{i+1}$ and extend $\gamma$ to $V^{n_{i+1}}$.
Since $f_2$ is subexponential, for $x\gg 1$ we have that $f_2(x)\leq 2^{\frac{x}{m_i+2}}$, so take $ n_{i+1} \geq m_i+2 $ such that $f_2(n_{i+1})\leq 2^{\frac{n_{i+1}}{m_i+2}}$. Moreover, take: $$ (*)\ \ d_{m_i+1}=d_{m_i+2}=\cdots = d_{n_{i+1}}=n_{i+1}+1.$$
Pick $v_1,v_2\in V^{m_i+2}$ whose images modulo $V^{m_i}$ are linearly independent and fix a complement: $$ V^{n_{i+1}}=V^{m_i}\oplus Fv_1\oplus Fv_2\oplus W. $$ Set $\gamma(v_1)=x_{m_i+1},\gamma(v_2)=x_{m_i+2}$ and $\gamma\left(W\right)=~0$.

It follows from $(*)$ that all length-$\leq n_{i+1}$ free monomials in $x_{m_i+1},x_{m_i+2}$ are linearly independent in $A$. Hence:
\begin{eqnarray*}
f_2(n_{i+1}) & \leq & 2^{\frac{n_{i+1}}{m_i+2}} \\
& \leq & \gamma\left( V^{m_i+2} \right)^{\frac{n_{i+1}}{m_i+2}} \\
& \leq & \dim_F W_{n_{i+1}} \\
& = & w_\gamma(n_{i+1})
\end{eqnarray*}
(we may assume that $ m_i+2 | n_{i+1} $ for simplicity) and the claim follows. \end{proof}

\section{Approximating growth functions by linear maps}

We start with the following construction, which appears in \cite{SmoktunowiczBartholdi}, brought here in a slightly improved version.

Let $f\colon \mathbb{N}\rightarrow \mathbb{N}$ be a non-decreasing function satisfying $f(2^{n+1})\leq f(2^n)^2$ for every $n\geq 0$. Let $X$ be a set of cardinality $f(1)$.
Define a sequence of sets as follows. Let $W(1)=X$.
Suppose that $W(2^n)\subseteq~X^{2^n}$ is given such that $|W(2^n)|\geq f(2^n)$ and pick an arbitrary subset $C(2^n)\subseteq W(2^n)$ of cardinality $\lceil \frac{f(2^{n+1})}{f(2^n)} \rceil$. This is indeed possible since $\frac{f(2^{n+1})}{f(2^n)}\leq f(2^n)\leq |W(2^n)|$.
Let $W(2^{n+1})=W(2^n)C(2^n)$.

Consider the following set of right-infinite words:
$$ \mathcal{S} = W(1)C(1)C(2)C(2^2)C(2^3)C(2^4)\cdots \subseteq X^{\infty} $$

Let $h(n)$ be the number of words of length at most $n$ which occur as factors of some infinite word in $\mathcal{S}$; we call them factors of $\mathcal{S}$.

\begin{lem} \label{SBmon}
Under the above notations, there exist $C,D>0$ such that: $$h(n)\leq Cn^3f(Dn)$$
Moreover, if $f(2^k)$ divides $f(2^{k+1})$ for all $k$ then: $$h(n)\leq Cn^2f(Dn)$$
\end{lem}

\begin{proof}
By \cite[Lemma~6.3]{SmoktunowiczBartholdi}, every length-$2^m$ word which factors $\mathcal{S}$ is a subword of some word from $W(2^m)C(2^m)\cup C(2^m)W(2^m)$. Each word in $W(2^m)C(2^m)\cup C(2^m)W(2^m)$ has length $2^{m+1}$, so the number of length-$2^m$ words, which is equal to $h'(2^m)=h(2^m)-h(2^m-1)$, is at least $\#W(2^m)$ and at most: $$ ( 2^m + 1 ) \cdot \#\left(W(2^m)C(2^m)\cup C(2^m)W(2^m)\right) $$
Since $\#W(2^{m+1})=\#W(2^m)\cdot \lceil \frac{f(2^{m+1})}{f(2^m)} \rceil$, it follows by induction that $f(2^m)\leq \#W(2^m)\leq 2^m f(2^m)$. Notice that if $f(2^k)$ divides $f(2^{k+1})$ for all $k$ then $\#W(2^m)=f(2^m)$.
Therefore:
$$ f(2^m) \leq h'(2^m) \leq (2^m+1) \cdot \left(2\cdot 2^mf(2^m) \cdot \lceil \frac{f(2^{m+1})}{f(2^m)} \rceil \right) \leq 2^{2m+3}f(2^{m+1}) $$
and if $f(2^k)$ divides $f(2^{k+1})$ for all $k$ then:
$$ f(2^m) \leq h'(2^m) \leq (2^m+1) \cdot \left(2\cdot f(2^m) \cdot \frac{f(2^{m+1})}{f(2^m)}\right) \leq 2^{m+2}f(2^{m+1}) $$
Since every monomial which factors $\mathcal{S}$ is right-extendable to a longer factor of $\mathcal{S}$, it follows that $h'$ is non-decreasing, so: $$ f(2^m)\leq h(2^m)\leq 2^m h'(2^m)\leq 2^{3m+3}f(2^{m+1}) $$ 
and in the case that $f(2^k)$ divides $f(2^{k+1})$ for all $k$ then moreover:
$$ f(2^m)\leq h(2^m)\leq 2^m h'(2^m)\leq 2^{2m+2}f(2^{m+1}). $$
Since $h,f$ are non-decreasing then for every $n$, if we take $2^m\leq n\leq 2^{m+1}$ then we obtain that $f(n)\leq f(2^{m+1}) \leq h(2^{m+1}) \leq h(2n)$ and $h(n)\leq h(2^{m+1})\leq 2^{3(m+1)+3}f(2^{m+2})\leq 64 n^3f(4n)$, and in the case that $f(2^k)$ divides $f(2^{k+1})$ for all $k$, we further have that 
$h(n)\leq h(2^{m+1})\leq 2^{2(m+1)+2}f(2^{m+2})\leq 16 n^2 f(4n)$.
\end{proof}

\begin{prop} \label{Realizing_transformation}
Let $f\colon\mathbb{N}\rightarrow\mathbb{N}$ be an increasing, submultiplicative function and let $\delta\colon \mathbb{N}\rightarrow \mathbb{N}$ be a non-decreasing function tending to infinity (perhaps very slowly; we always assume $\delta(n)\leq n$).

Let $B$ be a finitely generated infinite-dimensional unital algebra.
Then there exists a countably generated, locally nilpotent algebra $A$ and a linear map $\gamma\colon B\rightarrow~A$ such that: 

$$f\left(\frac{n}{\delta(n)}\right)\preceq w_\gamma(n)\preceq n^5 f(n).$$
\end{prop}

\begin{proof}
Let $f\colon\mathbb{N}\rightarrow\mathbb{N}$ be an increasing, submultiplicative function and let $d=f(1)$.


For each $k\in \mathbb{N}$ let $X_k=\{x_{k,1},\dots,x_{k,d}\}$ and for an arbitrary subset $S\subseteq X_k^n$ denote by $S^{[m]}$ the subset of $X_{k+m}^n$ obtained by substituting each letter $x_{k,i}$ by $x_{k+m,i}$. Let $X=\bigcup_{k=1}^{\infty} X_k$. We call $k$ the \textit{level} of $x_{k,i}$.

Define a sequence of sets as follows. Let $W(1)=X_1$.
Suppose that $W(2^n)\subseteq~X_1^{2^n}$ has been defined such that $|W(2^n)|\geq f(2^n)$ and pick an arbitrary subset $C(2^n)\subseteq W(2^n)$ of cardinality $\lceil \frac{f(2^{n+1})}{f(2^n)} \rceil$ (as before, this is indeed possible since $f(2^{n+1})^2\leq f(2^n)^2$). Let $W(2^{n+1})=W(2^n)C(2^n)$.
Consider the following set of right-infinite words:
$$ \mathcal{S} = W(1)C(1)C(2)C(2^2)C(2^3)C(2^4)\cdots \subseteq X_1^{\infty} $$
Let: $$ A_1 = F\left<x_{1,1},\dots,x_{1,d}\right>_{\geq 1} / \left< w\ |\ w\ \text{is not a factor of}\ \mathcal{S} \right> $$
and let $h(n)$ be the number of non-zero monomials of length at most $n$ in $A_1$. By Lemma \ref{SBmon}, it follows that $h(n)\leq Cn^3f(Dn)$ for suitable constants $C,D>0$.

Let $\delta\colon \mathbb{N}\rightarrow\mathbb{N}$ be a non-decreasing function tending to infinity. Fix a sequence $1\leq n_1\leq n_2\leq \cdots$ of positive integers such that for every $r\in \mathbb{N}$:
$$ \min \{ k\ |\ r\leq n_k-\log_2 k \} \leq \delta(r) $$
which is indeed possible, since $\delta(r)\xrightarrow{r\rightarrow \infty} \infty$, e.g. take $n_k=2\max\{i|\delta(i)\leq k\}$.
Consider the following set of right-infinite words:

\begin{eqnarray*}
\mathcal{T} & = & W(1)C(1)\cdots C(2^{n_1}) \\
& \cdots & C(2^{n_1+1})^{[1]}\cdots C(2^{n_2})^{[1]} \\ 
& \vdots & \\
& \cdots & C(2^{n_i+1})^{[i]}\cdots C(2^{n_{i+1}})^{[i]}\cdots \subseteq X^{\infty}
\end{eqnarray*}
Let $\pi\colon \mathcal{T}\rightarrow \mathcal{S}$ be the function given by substituting each $x_{i,j}$ by $x_{1,j}$.
Consider the countably generated non-unital free algebra $F\left<X\right>_{\geq 1}$, and let: $$A=F\left<X\right>_{\geq 1} / \left< w\ |\ w\ \text{is not a factor of}\ \mathcal{T} \right>$$
For simplicity, we identify each $x_{i,j}$ with its image in $A$. Observe that $A$ is locally nilpotent. 

Let $B$ be a finitely generated infinite-dimensional algebra, generated by a finite-dimensional subspace $V$ such that $1\in V$ and $\dim_F V^{i+1} - \dim_F V^i \geq d$. This is indeed possible, enlarging $V$ if necessary.
Define a linear map $\gamma\colon B\rightarrow A$ as follows. Fix decompositions of $F$-vector spaces:
\begin{eqnarray*}
V & = & Fa_{1,1}\oplus\cdots \oplus Fa_{1,d}\oplus W_0 \\
\text{For}\ i\geq 1,\ \ \ V^{i+1} & = & V^i \oplus Fa_{i+1,1}\oplus\cdots \oplus Fa_{i+1,d} \oplus W_i
\end{eqnarray*}
Define $\gamma(a_{i,j})=x_{i,j}$ and $\gamma\left(W_{i-1}\right)=0$ for all $i\geq 1$ and $1\leq j\leq d$.

We now turn to estimate $w_\gamma (n)$.
Pick a finite factor $u$ of $\mathcal{T}$. Observe that it is uniquely determined by $\pi(u)$ together with the level of its first letter and the displacement of the first letter of the next level within $u$ (if such a letter appears in $u$). 

Notice that $W_n = \sum_{j_1+\cdots+j_s\leq n} \gamma\left(V^{j_1}\right)\cdots \gamma\left(V^{j_s}\right)$ is spanned by monomials in $A$ of length at most $n$. Moreover, the level of every letter in any of these monomials is at most $n$. 

Therefore, \begin{eqnarray*} (*)\ \ \ \ \ w_\gamma(n) & = & \dim_F W_n \\
& \leq & \#\{\text{Length-}\leq n\ \text{mmonomials in}\ \mathcal{S}\} \\
& \cdot & \#\{\text{Level of first letter}\} \\
& \cdot & \#\{\text{Displacement of letter with next level (if exists)}\} \\
& \leq & h(n)\cdot n \cdot n\leq Cn^5 f(Dn).\end{eqnarray*}

We now turn to bound $w_\gamma(n)$ from below. Fix a power of $2$, say, $2^m$.
Take $k\leq~\delta(m)$ such that $m\leq n_k-\log_2 k$, which is indeed possible by the way we have chosen $\{n_k\}_{k=1}^{\infty}$. In particular, $2^m\leq 2^{n_k}$; observe that every length-$2^m$ prefix of a word in $\mathcal{T}$ is a prefix of a words from $W(1)C(1)\cdots C(2^{n_k})^{[k-1]}$. Thus every length-$2^m$ prefix of a word in $\mathcal{T}$ belongs to $W_{k\cdot 2^m}$ (since the level of each letter is at most than $k$). By `forgetting' the levels (namely, applying $\pi$) we get:

\begin{eqnarray*}
f(2^m) & \leq & f(1)\lceil \frac{f(2)}{f(1)}\rceil \lceil \frac{f(2^2)}{f(2)}\rceil \cdots \lceil \frac{f(2^m)}{f(2^{m-1})}\rceil \\ & = & \# W(1)C(1)C(2)\cdots C(2^{m-1}) \\ & \leq &  \dim_F W_{k\cdot 2^m}
\end{eqnarray*}
Hence:
$$ f(2^m) \leq w_\gamma(k\cdot 2^m) \leq w_\gamma(\delta(m) 2^m) $$
Let $n$ be arbitrary and let $2^m\leq n\leq 2^{m+1}$. Then, since $f,w_\gamma,\delta$ are monotonely non-decreasing:
\begin{eqnarray*}
(**)\ \ \ \ \ f(n) & \leq & f(2^{m+1}) \\ & \leq & w_\gamma(\delta(m+1)\cdot 2^{m+1}) \\ & \leq & w_\gamma(2\delta(n)n)
\end{eqnarray*}
Let $N$ be given. Let $n=\lfloor \frac{N}{\delta(N)} \rfloor$. Then $\delta(n)n\leq \delta(N)\cdot \lfloor \frac{N}{\delta(N)} \rfloor \leq N$.
Then by $(*),(**)$:
\begin{eqnarray*}
f\left( \lfloor\frac{N}{\delta(N)}\rfloor \right) & = & f(n) \\ & \leq & w_\gamma(2\delta(n)n) \\ & \leq & w_\gamma(2N) \\ & \leq & 32CN^5 f(2DN),
\end{eqnarray*}
and the proof is completed.
\end{proof}

\section{Growth of nil algebras}

\subsection{Nil algebras with oscillating growth}

\begin{proof}[{Proof of Theorem B1}]
Let $f\colon \mathbb{N}\rightarrow \mathbb{N}$ be a subexponential function. Let $F$ be a countable field. Let $B$ be a finitely generated nil $F$-algebra of Gel'fand-Kirillov dimension at most $3$, which exists by \cite{LenaganSmoktunowiczYoung}. By Proposition \ref{Oscillating_transformation}, there exists a locally nilpotent algebra $A$ and a linear map $\gamma\colon B'\rightarrow A$ such that $w_\gamma(n)\geq f(2n)$ infinitely often and $w_\gamma(n)\leq \log_2 n$ infinitely often. Let $C$ be the subalgebra of $A\wr B$ generated by $B$ and $c_\gamma$. By \cite[§4]{AAJZ_TransAMS}, $C$ is nil.
By Lemma \ref{growth_wreath} we have that $g_C(n)\leq g_B(n)^2w_\gamma(n)+g_B(n)$ so infinitely often $g_C(n)\leq n^{6+\varepsilon}$, for every $\varepsilon>0$. In addition, we have $w_\gamma(n)\leq g_C(2n)$, so infinitely often $f(2n)\leq g_C(2n)$.
\end{proof}

\begin{proof}[{Proof of Theorem B2}]

Let $f\colon \mathbb{N}\rightarrow~\mathbb{N}$ be a subexponential function and let $\omega\colon \mathbb{N}\rightarrow~\mathbb{N}$ be a non-decreasing super-polynomial function.

Let $F$ be an arbitrary field. Let $B$ be a finitely generated nil $F$-algebra whose growth function satisfies $g_B(n)\leq \omega(n)^{1/3}$; this is indeed possible by \cite{BellYoung}.
By Proposition \ref{Oscillating_transformation}, there exists a locally nilpotent algebra $A$ and a linear map $\gamma\colon B'\rightarrow A$ such that $w_\gamma(n)\geq f(2n)$ infinitely often and $w_\gamma(n)\leq \log_2 n$ infinitely often. Let $C$ be the subalgebra of $A\wr B$ generated by $B$ and $c_\gamma$. By \cite[§4]{AAJZ_TransAMS}, $C$ is nil.
By Lemma \ref{growth_wreath} we have that: $$g_C(n)\leq g_B(n)^2w_\gamma(n)+g_B(n).$$
Therefore, infinitely often: $$g_C(n)\leq \omega(n)^{2/3}\log_2 n + \omega(n)^{1/3} \leq \omega(n)$$ where the last equality holds for $n\gg 1$.
In addition, we have $w_\gamma(n)\leq~g_C(2n)$, so infinitely often $f(2n)\leq g_C(2n)$.
\end{proof}

\begin{rem}
Since in Proposition \ref{Oscillating_transformation} each $n_k$ can be chosen to be arbitrarily large compared to $\{m_i,n_i\}_{i=1}^{k}$ and each $m_k$ can be chosen arbitrarily large compared to $\{m_i,n_i\}_{i=1}^{k}\cup \{n_k\}$, we can clearly make sure that in Theorems B1, B2 (respectively) the \textit{upper densities} satisfy:

$$ \limsup_{N\rightarrow \infty} \frac{\#\{k\in [1,N]\ |\ g_C(n)\leq n^{6+\varepsilon}\}}{N} = 1, $$
respectively:
$$ \limsup_{N\rightarrow \infty} \frac{\#\{k\in [1,N]\ |\ g_C(n)\leq \omega(n)\}}{N} = 1 $$
and (in both B1, B2):
$$ \limsup_{N\rightarrow \infty} \frac{\#\{k\in [1,N]\ |\ g_C(n)\geq f(n)\}}{N} = 1 $$
Thus, considering the values of $g_C(n)$ for $n$ ranging up to a given value, it simultaneously looks arbitrarily fast and arbitrarily slow.
\end{rem}

\subsection{Realizing growth functions of nil algebras}

\begin{proof}[{Proof of Theorem A1}]
Let $F$ be a countable field and let $B$ be a finitely generated nil $F$-algebra of Gel'fand-Kirillov dimension at most $3$, which exists by \cite{LenaganSmoktunowiczYoung}. In particular, for every $\varepsilon>0$, $g_B(n)\preceq n^{3+\varepsilon}$. Let $f\colon \mathbb{N}\rightarrow \mathbb{N}$ be an increasing, submultiplicative function and let $\delta\colon \mathbb{N}\rightarrow \mathbb{N}$ be an arbitrarily slow non-decreasing function tending to infinity. Then by Proposition \ref{Realizing_transformation} there exists a locally nilpotent algebra $A$ and a linear map $\gamma\colon B'\rightarrow A$ such that:
$$f\left(\frac{n}{\delta(n)}\right)\preceq w_\gamma(n)\preceq n^5 f(n)$$
Let $C$ be the subalgebra of $A\wr B$ generated by $B$ and $c_\gamma$. By \cite[§4]{AAJZ_TransAMS}, $C$ is nil.
By Lemma \ref{growth_wreath}, for every $\varepsilon>0$ we have:
\begin{eqnarray*}
f\left(\frac{n}{\delta(n)}\right) \preceq w_\gamma(n) & \preceq & g_C(n) \\
& \preceq & g_B(n)^2w_\gamma(n) \preceq n^{11+\varepsilon}f(n)
\end{eqnarray*}
as claimed.
\end{proof}

\begin{proof}[{Proof of Theorem A2}]
Let $F$ be an arbitrary field. Let $f\colon \mathbb{N}\rightarrow \mathbb{N}$ be an increasing, submultiplicative function and let $\omega\colon \mathbb{N}\rightarrow \mathbb{N}$ be a non-decreasing, super-polynomial function.
Let $\delta\colon \mathbb{N}\rightarrow \mathbb{N}$ be an arbitrarily slow non-decreasing function tending to infinity.

Let $B$ be a finitely generated nil $F$-algebra whose growth satisfies $g_B(n)\preceq \omega(n)^{1/3}$; this is indeed possible by \cite{BellYoung}.
By Proposition \ref{Realizing_transformation} there exists a locally nilpotent algebra $A$ and a linear map $\gamma\colon B' \rightarrow A$ such that:
$$f\left(\frac{n}{\delta(n)}\right)\preceq w_\gamma(n)\preceq n^5 f(n)$$
Let $C$ be the subalgebra of $A\wr B$ generated by $B$ and $c_\gamma$. By \cite[§4]{AAJZ_TransAMS}, $C$ is nil.
By Lemma \ref{growth_wreath}, we have:
\begin{eqnarray*}
f\left(\frac{n}{\delta(n)}\right) \preceq w_\gamma(n) & \preceq & g_C(n) \\
& \preceq & g_B(n)^2w_\gamma(n) \preceq \omega(n)^{2/3}n^5 f(n)<\omega(n)f(n)
\end{eqnarray*}
where the last inequality hold for $n\gg 1$, and the proof is completed.
\end{proof}

\section{Applications and specifications}

\subsection{Nil Lie algebras of oscillating growth}
Our next application is a construction of a finitely generated nil restricted Lie algebras whose growth oscillates between a polynomial and an arbitrarily rapid (subexponential) function.

In \cite{PetrogradskyLie} Petrogradsky constructed nil restricted Lie algebras with oscillating growth over fields of positive characteristic; our constructions are not subject to this restriction. Moreover, the upper bounds in our constructions are arbitrarily rapid (subexponential), unlike the upper bounds in \cite{PetrogradskyLie}. However, our lower bounds -- though polynomial, at least in the countable case -- do not get close to linear as Petrogradsky's examples. This is inevitable, since they share their growth rates with associative algebras.

\begin{cor} \label{oscillatingLie}
Let $f\colon\mathbb{N}\rightarrow\mathbb{N}$ be an arbitrarily rapid subexponential function.
Let $F$ be a countable field of $\charac(F)\neq 2$. Then there exists a finitely generated nil Lie algebra $L$ over $F$ such that:
\begin{itemize}
\item $g_L(n)\leq n^{6+\varepsilon}$ infinitely often for every $\varepsilon>0$; and
\item $g_L(n)\geq f(n) $ infinitely often.
\end{itemize}
\end{cor}

\begin{cor} \label{oscillatingLie2}
Let $f\colon\mathbb{N}\rightarrow\mathbb{N}$ be an arbitrarily rapid subexponential function and let $\omega\colon \mathbb{N}\rightarrow \mathbb{N}$ be an arbitrarily slow, non-decreasing, super-polynomial function.

Let $F$ be an uncountable field of $\charac(F)\neq 2$. Then there exists a finitely generated nil Lie algebra $L$ over $F$ such that:
\begin{itemize}
\item $g_L(n)\leq \omega(n) $ infinitely often; and
\item $g_L(n)\geq f(n) $ infinitely often.
\end{itemize}
\end{cor}

\begin{proof}[{Proof of Corollaries \ref{oscillatingLie},\ref{oscillatingLie2}}]
By \cite{AlahmadiAlharthi}, if $A$ is a nil associative algebra over a field of characteristic $\neq 2$ then the Lie algebra $L=[A,A]$ is finitely generated of growth $g_L(n)\sim g_A(n)$. It is easy to see that we may assume that: $$ g_L(n)\leq g_A(n),\ g_A(n)\leq Rg_L(Rn) $$ for some $R\in \mathbb{N}$ (enlarging the generating subspace of $A$ is necessary). Now use Theorems B1, B2 with $\bar{f}(n)$ defined as follows.

We may assume that $f$ is non-decreasing. Since $\frac{1}{n}\log_2 f(n)\xrightarrow{n\rightarrow \infty} 0$, we can find a non-decreasing subexponential function $\theta\colon \mathbb{N}\rightarrow \mathbb{N}$ such that $\theta(n)\xrightarrow{n \rightarrow \infty}~ \infty$ 
and $\theta(n)\cdot~ \frac{1}{n}\log_2 f(n)\rightarrow 0$.
Then the function $\bar{f}(n):=f(n\theta(n))\theta(n)$ is still subexponential:
\begin{eqnarray*}
\frac{1}{n}\log_2 \left(  f(n\theta(n))\theta(n) \right) & = & \frac{1}{n}\log_2 \theta(n) + \frac{1}{n}\log_2 f(n\theta(n)) \\ & = & o(1) + \frac{\theta(n)}{n\theta(n)}\log_2 f(n\theta(n)) \\ & \leq & o(1) +  \frac{\theta(n\theta(n))}{n\theta(n)}\log_2 f(n\theta(n)) = o(1)
\end{eqnarray*}

Hence if $g_A(n)\geq \bar{f}(n)$ and $n\gg 1$ then: $$ g_L(Rn)\geq \frac{1}{R} g_A\left( n \right)\geq \frac{1}{R}\bar{f}(n) = \frac{1}{R}\theta(n) f\left( n \theta(n)\right)\geq f(Rn) $$ so $g_L(n)\geq f(n)$ infinitely often.
In addition, $g_L(n)\leq g_A(n)$ so $g_L(n)\leq n^{6+\varepsilon}$ (in the countable case) and $g_L(n)\leq \omega(n)$ (in the arbitrary case) infinitely often.

Since $A$ is nil, it is evident that $x\mapsto [a,x]$ is nilpotent for any $a\in L$, so $L$ is a nil Lie algebra.
\end{proof}

\subsection{Nil algebras of intermediate growth}

In \cite{Smoktunowicz}, Smoktunowicz constructed the first nil algebra of intermediate growth. The example, obtained as a suitable quotient of a nil Golod-Shafarevich algebra, is shown to have intermediate growth using the main theorem therein \cite[Theorem~A]{Smoktunowicz}, which gives (very technical) upper and lower bounds on the growth of graded algebras with sparse homogeneous relations; in particular, Smoktunowicz is able to construct a nil algebra \cite[Corollary~4.5]{Smoktunowicz} whose growth is smaller than $n^{c\ln^2 n}$ but infinitely often greater than $n^{d\ln n}$ for suitable constants $c,d>0$.

Theorems A1, A2 proven here enable one to construct nil algebras with a much finer control on their growth rates, thereby establishing a significant step toward proving the conjecture that any (non-linear) function which occurs as the growth rate of an algebra is realizable as the growth of a nil algebra (see \cite{AAJZ_ERA}).
For example, one can construct nil algebras whose growth is bounded between $n^{c\ln^\alpha n}$ and $n^{d\ln^\alpha n}$ for an arbitrary $\alpha\in (0,\infty)$ (and some $c,d>0$); and nil algebras with arbitrary GK-superdimension, namely, whose growth functions are between $\exp(n^{\alpha-\varepsilon})$ and $\exp(n^{\alpha+\varepsilon})$ for all $\varepsilon>0$ and $\alpha\in (0,1)$ arbitrary.

Another evidence for the richness of the variety of possible growth functions of nil algebras that Theorems A1, A2 afford is given by the notion of `$q$-dimensions', introduced by Petrogradsky in his analysis of the growth of universal enveloping algebras of Lie algebras \cite{Petrogradsky96,Petrogradsky00,Petrogradsky93}.
More specifically, consider the following hierarchy of growth functions for $q=1,2,3,\dots$ and $\alpha\in \mathbb{R}^{+}$:
$$ \Phi_\alpha^{1}(n) = \alpha,\ \Phi_\alpha^{2}(n) = n^\alpha,\ \Phi_\alpha^{3}(n) = \exp(n^{\alpha/\alpha+1}),\ \text{and}\ \Phi_\alpha^{q>3}(n) = \exp\left(\frac{n}{(\ln^{(q-3)} n)^{1/\alpha}}\right)\ $$
where $\ln^{(1)} n=\ln n,\ \ln^{(q+1)} n=\ln(\ln^{(q)} n)$.
For an algebra $A$ whose growth function is $g_A(n)$, define: 
\begin{eqnarray*}
\operatorname{Dim}^q(A) & := & \inf\{\alpha\in\mathbb{R}^{+}|\ g_A(n)\leq \Phi_\alpha^q(n)\ \forall n\gg 1\} \\ & = & \sup\{\alpha\in\mathbb{R}^{+}|\ g_A(n)\geq \Phi_\alpha^q(n)\ \forall n\gg 1\}
\end{eqnarray*}
if they are equal to each other. Petrogradsky proved that if $\operatorname{Dim}^q(\mathfrak{g})=\alpha$ then $\operatorname{Dim}^{q+1}(U(\mathfrak{g}))=~\alpha$. He left open the problem of whether every (a priori possible) real number is realizable as the $q$-dimension of some associative/Lie algebra (see \cite[Remark in page 347]{Petrogradsky00} and \cite[Remark in page 651]{Petrogradsky97}). This is indeed true, and even within the classes of finitely generated simple associative and Lie algebras (see \cite{SimpleLie}).
\begin{cor} \label{intermediateLieapp}
Theorems A1, A2 remain true if $R$ is replaced by a finitely generated nil Lie algebra.
\end{cor}

\begin{proof}
This is immediate by \cite{AlahmadiAlharthi}.
\end{proof}

Using Theorems A1, A2 for Lie algebras, it is easy to construct nil associative and nil Lie algebras with arbitrary $q$-dimensions as well (e.g. taking $\delta(n)=\log_2 n,\ \omega(n)=n^{\log_2 n}$).







\section{Domains with oscillating growth}

We now construct examples of noncommutative domains of oscillating growth, in a similar spirit of the Kassabov-Pak construction \cite{KassabovPak}. We start with another construction of Lie algebras of oscillating growth, whose behavior is stricter than of the aforementioned constructions. We then consider their universal enveloping algebras and utilize a finitary version of the connection between the growth of $L$ and $U(L)$ studied by Smith (and significantly extended by Petrogradsky).

\begin{lem} \label{lemdom}
Let $f\colon \mathbb{N}\rightarrow \mathbb{N}$ be a non-decreasing subexponential function.
Let: $$1<n_1<n'_1<m_1<m'_1<n_2<n'_2<\cdots$$ be a sequence of positive integers, arbitrary subject to the assumption that each integer is large enough than its predecessor.

Then there exists a finitely generated associative algebra $A$, a finitely generated Lie algebra $L$ and a constant $c\in \mathbb{N}$ such that for every $i\gg 1$:
\begin{itemize}
    \item If $k\in [2^{n_i+c},2^{n'_i}]$ then $g_A(k),g_L(k)\geq f(k)$
    \item If $k\in [2^{m_i},2^{m'_i}]$ then $g_A(k),g_L(k)\leq k^2\log_2 k$
\end{itemize}
\end{lem}
\begin{proof}
We start by defining a function $g\colon\mathbb{N}\rightarrow \mathbb{N}$, first on powers of $2$, then extending $g$ to $\mathbb{N}$.
There exists $r$ such that for all $s\geq r$ we have $f(2^s)\leq 2^{2^s}$, since $f$ is subexponential. Fix such $n_1\geq r$ for such an $r$ and set: $$g(1)=2,g(2)=2^2,\dots, g(2^{n'_1})=2^{2^{n'_1}}$$
then it follows that $g(2^k)\geq f(2^k)$ for all $k\in [n_1,n'_1]$.

Suppose that $n'_i$ has been defined, along with the values of $g$ for all powers of $2$ until $2^{n'_i}$. Let $c=g(2^{n'_i})$. Take $m_i\geq 2^c$ and set $g(2^{t+1})=g(2^t)$ for all $n'_i\leq t< m'_i$. It follows that for all $k\in [m_i,m'_i]$ we have $g(2^k)=c \leq \log_2 m_i\leq \log_2 k$.

Now suppose that $m'_i$ has been defined, along with the values of $g$ for all powers of $2$ until $2^{m'_i}$. Let $M=g(2^{m'_i})$.
We know that $f$ is subexponential, so there exists $r$ such that for all $s\geq r$ we have $f(2^s)\leq \left(M^{1/2^{m'_i}}\right)^{2^s}$. Take $n_{i+1}\geq r$ for such an $r$.
Define: $g(2^{t+1})=g(2^t)^2$ for all $m'_i\leq t<n'_{i+1}$. In particular,
$$ g(2^k) = M^{2^{k-m'_i}} = \left(M^{1/{2^{m'_i}}}\right)^{2^{k}} \geq f(2^k) $$
for all $k\in [n_{i+1},n'_{i+1}]$.
We now extend $g$ to $\mathbb{N}$ by setting:
\begin{itemize}
    \item $g(k)=g\left(2^{\lceil \log_2 k \rceil}\right)$ if $k\in [2^{n_i},2^{n'_i}]$ for some $i\in \mathbb{N}$
    \item $g(k)=g\left(2^{\lfloor \log_2 k \rfloor}\right)$ otherwise.
\end{itemize}

It is clear by construction that $g \colon \mathbb{N} \rightarrow \mathbb{N}$ is non-decreasing and $g(2^{k+1})\leq g(2^k)^2$ for all $k\in \mathbb{N}$. Moreover, each $g(2^k)$ divides $g(2^{k+1})$ and it is straightforward that for every $i\in \mathbb{N}$: 
\begin{itemize}
    \item $g(k)\geq f(k)$ for all $k\in [2^{n_i},2^{n'_i}]$
    \item $g(k)\leq \log_2 \log_2 k$ for all $k\in [2^{m_i},2^{m'_i}]$
\end{itemize}

By Lemma \ref{SBmon}, there exists a finitely generated monomial algebra $A$ such that $g(n)\leq g_A(n) \leq Cn^2g(Dn)$ for all $n\in \mathbb{N}$ and some constants $C,D>0$.

It follows that for every $i\in \mathbb{N}$: 
\begin{itemize}
    \item $g_A(n) \geq g(k)\geq f(k)$ for all $k\in [2^{n_i},2^{n'_i}]$
    \item $g_A(k) \leq C'k^2\log_2\log_2 k$ for all $k\in [2^{m_i},2^{m'_i}]$
\end{itemize}

We recall the construction from Lemma \ref{SBmon}, to establish an important property of it: we construct a system of subsets $W(1)=\{x_1,x_2\}$, $C(2^n)\subseteq W(2^n)$ and $W(2^{n+1})=W(2^n)C(2^n)$.
Notice that since $g(2^i)=2^{2^i}$ for all $1\leq i\leq n'_1$, there must exist a word $\rho\in W(2^{n'_1})$ which involves both $x_1,x_2$.
But $g(2^{n'_1+1})=g(2^{n'_1})$, so by construction $C(2^{n'_1})$ is a singleton. We choose it to be $\{\rho\}$. It follows that every sufficiently long monomial in the constructed algebra $A$ contains a copy of $\rho$ as a factor, so in particular $x_1$ and $x_2$ are nilpotent.

Let $L=[A,A]$ be the commutator Lie algebra associated with $A$. By \cite{AlahmadiAlharthi} it follows that $L$ is a finitely generated Lie algebra whose growth is equivalent to that of $A$, and (enlarging the generating subspace of $A$ if needed) we may assume that: $$g_A(n) \leq Rg_{L}(Rn),\ g_L(n)\leq g_A(n),$$ for some $R\in \mathbb{N}$.
Replacing $f$ by $\bar{f}(n)=\theta(n) f(n\theta(n))$ which is still subexponential, and $\theta(n)\rightarrow \infty$ is a non-decreasing unbounded function (as in the proof of Corollaries \ref{oscillatingLie}, \ref{oscillatingLie2}), we get that for $i\gg 1$ if $k\in [2^{n_i+c},2^{n'_i}]$ for $c>\log_2 R$:
$$ g_L(k)\geq \frac{1}{R}g_A\left(\lfloor \frac{k}{R} \rfloor \right) \geq \frac{1}{R} \bar{f}\left(\lfloor \frac{k}{R} \rfloor \right) = \frac{\theta\left(\lfloor k/R \rfloor\right)}{R}f\left(\lfloor \frac{k}{R}\rfloor\theta\left(\lfloor \frac{k}{R}\rfloor\right)\right) \geq f(k)$$
Therefore for $i\gg 1$:
\begin{itemize}
    \item $g_{L}(k)\geq f(k)$ for all $k\in [2^{n_i+c},2^{n'_i}]$
    \item $g_{L}(k)\leq g_A(k)\leq k^2\log_2 k$ for all $k\in [2^{m_i},2^{m'_i}]$
\end{itemize}
The proof is completed.
\end{proof}


\begin{proof}[{Proof of Theorem C}]
Given: $$1<n_1<n'_1<m_1<m'_1<n_2<n'_2<\cdots$$ such that each integer is larger enough than its predecessor in a sense that will be made clear later, let $L$ be a finitely generated Lie algebra $L$ generated by a space $V$ such that for every $i\gg 1$:
\begin{itemize}
    \item If $k\in [2^{n_i+c},2^{n'_i}]$ then $g_{L}(k)\geq f(k)$
    \item If $k\in [2^{m_i},2^{m'_i}]$ then $g_{L}(k)\leq k^2\log_2 k$
\end{itemize}
For some $c\in \mathbb{N}$. This is done using Lemma \ref{lemdom}. Let $U=U(L)$ be the universal enveloping algebra associated with $L$. Let $V'$ be the subspace of $U$ spanned by $1$ and the image of $V$ in $U$. The associated growth function satisfies $g_{U}(n)\geq g_{L}(n)$, so $g_{U}(n)\geq f(n)$ infinitely often. 


Let $\{a_n\}_{n=1}^{\infty}, \{b_n\}_{n=1}^{\infty}$ be sequences of positive integers such that: $$ (1)\ \ \ \prod_{n=1}^{\infty} \frac{1}{(1-t^n)^{b_n}} = \sum_{n=0}^{\infty} a_nt^n $$
Petrogradsky proved \cite[Theorem~1.2]{Petrogradsky00} that if $b_n=n^{\alpha-1+o(1)}$ then $a_n=\exp\left(n^{\frac{\alpha}{\alpha+1}+o(1)}\right)$.
By the Poincar\'e-Birkhoff-Witt Theorem (see \cite[Page 340]{Petrogradsky00}), the sequences $a_n=\gamma'_U(n),\ b_n=~\gamma'_L(n)$ satisfy $(1)$, namely:
$$ (2)\ \ \ \prod_{n=1}^{\infty} \frac{1}{(1-t^n)^{g_L'(n)}} = \sum_{n=0}^{\infty} g_U'(n)t^n $$
It is easy to derive a finitary version, namely:

\begin{eqnarray*}
g_U'(N) & = & \coeff_{t^N}\left( \prod_{n=1}^M \frac{1}{(1-t^n)^{g_L'(n)}} \right)
\end{eqnarray*}
for any $M\geq N$. For any $0\leq s\leq 2^{m'_i}$:
$$ (3)\ g_U'(s) = \coeff_{t^{s}}\left( \prod_{n=1}^{2^{m'_i}} \frac{1}{(1-t^n)^{g_L'(n)}} \right) = \coeff_{t^s} \left( \prod_{n=1}^{2^{m_i}} \frac{1}{(1-t^n)^{g_L'(n)}} \cdot \prod_{n=2^{m_i}+1}^{2^{m'_i}} \frac{1}{(1-t^n)^{g_L'(n)}} \right) $$


We need the following probably well-known facts.

\begin{lem} \label{PowerSeries}
Let $f,g\in \mathbb{Z}_{\geq 0}[[t]]$. Then for each $R\geq 0$:
\begin{itemize}
    \item If $\coeff_1(g)>0$ then $\coeff_{t^R} (fg) \geq \coeff_{t^R} (f)$
    \item If $\coeff_{t^k} (f) \leq T$ for all $0\leq k\leq R$ then $\coeff_{t^R} (fg)\leq T\cdot \sum_{k=0}^{R} \coeff_{t^k} g$
    \item $\coeff_{t^k} \prod_{i=1}^{m} \frac{1}{(1-t^i)^{d_i}} \leq R^{\sum_{j=1}^{m} d_j}$ for all $0\leq k\leq R$
\end{itemize}
\end{lem}
\begin{proof}
\begin{itemize}
    \item $\coeff_{t^R}(fg) = \sum_{i=0}^{R} \coeff_{t^i}(f)\cdot \coeff_{t^{R-i}}(g)\geq \coeff_{t^R}(f)\cdot \coeff_1(g)\geq \coeff_{t^R}(f)$
    \item $\coeff_{t^R}(fg)=\sum_{k=0}^{R} \coeff_{t^k}(f)\cdot \coeff_{t^{R-k}}(g) \leq T\cdot \sum_{k=0}^{R} \coeff_{t^k}(g)$
    \item Suppose that $f\in \mathbb{Z}_{\geq 0}[[t]]$ satisfies $\coeff_{t^k}(f)\leq c$ for each $0\leq k\leq R$, then $\coeff_{t^k}(f\cdot (1+t^j+t^{2j}+\cdots))\leq \sum_{i=0}^{k} \coeff_{t^i}(f) \leq ck \leq cR$ for each $0\leq k\leq R$. The claim now follows by induction, since $\prod_{i=1}^{m} \frac{1}{(1-t^i)^{d_i}}$ is a product of $\sum_{j=1}^{m} d_j$ functions of the form $1+t^j+t^{2j}+\cdots$.
\end{itemize}
\end{proof}

Back to $(3)$, using Lemma \ref{PowerSeries}, for all $0\leq s\leq 2^{m'_i}$:
\begin{eqnarray*} (4)\ \ 
g'_U(s) & = & \coeff_{t^{s}}\left( \prod_{n=1}^{2^{m_i}} \frac{1}{(1-t^n)^{g_L'(n)}} \cdot \prod_{n=2^{m_i}+1}^{2^{m'_i}} \frac{1}{(1-t^n)^{g_L'(n)}} \right) \\
& \leq & \left(2^{m'_i}\right)^D \cdot \sum_{r=0}^{s} \coeff_{t^r} \left(\prod_{n=2^{m_i}+1}^{2^{m'_i}} \frac{1}{(1-t^n)^{g_L'(n)}}\right) \\
& \leq & \left(2^{m'_i}\right)^D \cdot \sum_{r=0}^{2^{m'_i}} \coeff_{t^r} \left(\prod_{n=2^{m_i}+1}^{2^{m'_i}} \frac{1}{(1-t^n)^{g_L'(n)}}\right) \\
& = & \left(2^{m'_i}\right)^D \cdot \left( 1 + \sum_{r=2^{m_i}+1}^{2^{m'_i}} \coeff_{t^r} \left(\prod_{n=2^{m_i}+1}^{2^{m'_i}} \frac{1}{(1-t^n)^{g_L'(n)}}\right) \right) 
\end{eqnarray*}
where $D$ depends only on $m_i$ and on the values of $g'_L$ until $2^{m_i}$.
Since $g_L'(k)\leq k^2\log_2 k$ for all $k\in [2^{m_i},2^{m'_i}]$, using Lemma \ref{PowerSeries}, for each $r$:
\begin{eqnarray*}
\coeff_{t^r} \left(\prod_{n=2^{m_i}+1}^{2^{m'_i}} \frac{1}{(1-t^n)^{g_L'(n)}}\right) & \leq & \coeff_{t^r} \left(\prod_{n=1}^{\infty} \frac{1}{(1-t^n)^{\lfloor n^2\log_2 n\rfloor)}}\right)
\end{eqnarray*}

By Petrogradsky's aforementioned analysis (for $\alpha=3)$ we get that, for any $\varepsilon>0$
there exists $p_\varepsilon$ such that for all $r\geq p_\varepsilon$:
$$ \coeff_{t^r} \left(\prod_{n=1}^{\infty} \frac{1}{(1-t^n)^{\lfloor n^2\log_2 n\rfloor)}}\right) \leq \exp\left(r^{\frac{3}{4}+\frac{\varepsilon}{3}}\right) $$
and we can make sure that $2^{m_i}\geq p_{1/i}$ for each $i$.
Then, by $(4)$, we have for all $0\leq s\leq 2^{m'_i}$:
$$ g'_U(s) \leq 2\cdot (2^{m'_i})^D\cdot 2^{m'_i}\cdot \exp\left((2^{m'_i})^{\frac{3}{4}+\frac{1}{3i}}\right) $$
so:
$$ g_U(2^{m'_i}) = \sum_{s=0}^{2^{m'_i}} g'_U(s) \leq (2^{m'_i}+1)\cdot 2\cdot (2^{m'_i})^D\cdot 2^{m'_i}\cdot \exp\left((2^{m'_i})^{\frac{3}{4}+\frac{1}{3i}}\right) \leq \exp\left((2^{m'_i})^{\frac{3}{4}+\frac{1}{i}}\right) $$
where the last inequality holds for $i\gg 1$ (and $m'_i\gg m_i$, as $D$ is independent of $m'_i$), and the claim follows.
\end{proof}

\section{Polynomial oscillations and Warfield's question on tensor products}


In this section we construct primitive monomial algebras whose tensor product has a prescribed GK-dimension, proving Theorem D. 
These are the first semiprime examples for which the GK-dimension is strictly subadditive with respect to tensor products, which was left open in \cite{KrempaOkninski} and \cite[Page~167]{KrauseLenagan}.
Notice that the first family of primitive monomial algebras of arbitrary GK-dimension has constructed by Vishne \cite{Vishne}.

\subsection{Monomial algebras, infinite words and complexity}

Let $\Sigma=\{x_1,\dots,x_d\}$ be a finite alphabet. Let $w\in \Sigma^{\mathbb{N}}$ be an infintie word:
$$ w = w_1w_2w_3\cdots $$
we say that a finite word over $\Sigma$, say $u=x_{i_1}\cdots x_{i_m}$, is a factor of $w$ if $u=w[k,k+m-1]=w_k\cdots w_{k+m-1}$ for some $k\in \mathbb{N}$. The complexity function $p_w(n)$ is defined by:
$$ p_w(n) = \# \{u\in \Sigma^n\ |\ u\ \text{is a factor of}\ w\} $$
An infinite word $w$ defines a monomial algebra generated by $x_1,\dots,x_d$:
$$ A_w = F\left<x_1,\dots,x_d\right> / \left< u\ |\ u\ \text{is not a factor of}\ w \right> $$
Many algebraic properties of $A_w$ reflect combinatorial and dynamical properties of $w$. For instance, $A_w$ is prime if and only if $w$ is recurrent, namely, every factor $u$ of $w$ has infinitely many occurrences (i.e. $u=w[k,l]$ for infinitely many $k,l$'s). An algebra is just-infinite if it is infinite-dimensional, but every proper homomorphic image of it is finite-dimensional. The algebra $A_w$ is just-infinite if and only if $w$ is uniformly recurrent \cite[Theorem~3.2]{BBL}, namely, for each factor $u$ of $w$ there exists $C_u>0$ such that any length-$C_u$ factor of $w$ contains an occurrence of $u$. 

An important class of infinite words is the class of Toeplitz sequences. These are infinite words $w$ for which for any $n\geq 1$ there exists $d\geq 1$ such that $w_n=w_{n+d}=w_{n+2d}=\cdots$.
The following is standard, but we bring a full proof here for the reader's convenience.

\begin{lem} \label{words}
Let $w$ be a Toeplitz sequence. If the monomial algebra $A_w$ has super-linear growth then it is just-infinite primitive.
\end{lem}
\begin{proof}
Suppose that $w\in \Sigma^{\mathbb{N}}$ is a Toeplitz sequence, and for each $n$ let $d_n$ be such that $w_n=w_{n+id_n}$ for all $i\geq 1$. Let $w[k,k+m-1]$ be a factor; then $w[k,m]=w[k+iD,m+iD]$ for all $i\geq 1$, where $D=\text{lcm}(d_k,d_{k+1},\dots,d_m)$. Thus every sufficiently long factor of $w$ contains an occurrence of $w[k,m]$, and $w$ is uniformly recurrent. By \cite[Theorem~3.2]{BBL}, $A_w$ is just-infinite. 

By \cite{Farina}, $A_w$ is prime. By \cite{Okn}, either $A_w$ is primitive, PI or has a non-zero Jacobson radical.
Since the growth of $A_w$ is super-linear, $w$ cannot be eventually periodic, hence $A_w$ cannot be PI (e.g. by \cite[3.2.2(3)]{Madill}). The Jacobson radical $J$ of $A_w$ is locally nilpotent by \cite{loc nilp jac rad}. If $J\neq 0$ then by just-infiniteness, $\dim_F A_w/J < \infty$. By \cite{Bergman}, $J$ is homogeneous, so $A_w/J$ is graded, finite-dimensional and semiprimitive. Thus $J=\left(A_w\right)_{\geq 1}$ is the augmentation ideal consisting of all elements with zero constant term; this ideal is finitely generated as an algebra (by the set of letters of the underlying alphabet), hence nilpotent, contradicting that $\dim_F A_w=\infty$. It follows that $A_w$ is primitive.
\end{proof}

Let $V=F+Fx_1+\cdots+Fx_d$ be the generating subspace of $A_w$ spanned by $1$ and by the letters of $\Sigma$. Then the growth function of $A_w$ can be interpreted via the complexity of $w$, namely, $$ g_{A_w,V}(n)=\dim_F V^n=\Span_F \{ u\in \Sigma^{\leq n}\ |\ u\ \text{is a factor of}\ w \} = \sum_{k=0}^{n} p_w(k) $$
(By default, we take $V$ as a generating subspace, and omit the subscript notation.) Notice that $p_w(n)$ is non-decreasing. In particular, $g_{A_w}(n)\leq np_w(n)$ and $np_w(n)\leq \sum_{k=n+1}^{2n} p_w(k)\leq g_{A_w}(cn)$ for any $c\geq 2$.

\subsection{Polynomial oscillations and growth of tensor products}

Fix $\alpha,\beta\geq 2$. Without loss of generality, suppose that $\alpha\leq \beta$. Let $2\leq \gamma\leq \alpha$.
Let us construct a sequence of positive integers:
$$ 1 < s < d_1 < e_1 < d_2 < e_2 < \cdots $$
where $s\gg_\beta 1$ and $e_i=2^{d_i},\ d_{i+1}=2^{e_i}$. Define sequences $(m_k)_{k\geq 1},(n_k)_{k\geq 1}$ of positive integers. Fix $t\gg_\beta 1$ and set $n_k=m_k=t$ for $1\leq k \leq s$. Formally set $d_0=e_0=s$.
\begin{itemize}
    \item For each $i\geq 0$ and $k\in (e_i,d_{i+1}]$, set $m_k=t$;
    \item For each $i\geq 1$ and $k\in [d_i,e_i)$, set $m_{k+1}=\lceil 5^{\beta-2} m_k \rceil$.
\end{itemize}
And:
\begin{itemize}
    \item For each $i\geq 1$,
    $n_{d_i+1}=t$ and for each $k\in (d_i,e_i)$, set $n_{k+1}=\lceil 5^{\gamma-2} n_k \rceil$;
    \item For each $i\geq 0$ and $k\in [e_i,d_{i+1})$, set $n_{k+1}=\lceil 5^{\alpha-2} n_k \rceil$.
\end{itemize}
Notice that $\frac{n_{k+1}}{n_k},\frac{m_{k+1}}{m_k}$ are bounded above by a constant $\lambda$ depending only on $\alpha,\beta,\gamma$. Since $s,t\gg_\beta 1$, the above definitions yield $m_{k+1}\leq (m_k-1)^2,\ n_{k+1}\leq (n_k-1)^2$, and $m_k,n_k\geq 3$.

Therefore, by \cite[Proposition~4.79]{Kurka} there exist Toeplitz sequences $X,Y$ whose complexity functions $p_X(n),p_Y(n)$ satisfy:
$$ 2\cdot 5^{k-1}\cdot n_k\leq p_X(5^k) \leq 2\cdot 5^k\cdot n_k $$
$$ 2\cdot 5^{k-1}\cdot m_k\leq p_Y(5^k) \leq 2\cdot 5^k\cdot m_k $$
(The statement of the \cite[Proposition~4.79]{Kurka} is different, but these bounds are explicitly proven in the proof.)
For $r\geq 5$ let $5^k\leq r<5^{k+1}$ (that is, $k=\lfloor \log_5 r \rfloor$) and observe that:
$$ (*)\ \ \ \frac{2}{25}r\cdot n_k\leq p_X(r) \leq 10r\cdot n_{k+1},\ \ \ \frac{2}{25}r\cdot m_k\leq p_Y(r) \leq 10r\cdot m_{k+1} $$
Associated with $X,Y$ are the corresponding monomial algebras $A_X,A_Y$ spanned by all finite factors of $X,Y$ (respectively). Let $g_{A_X},g_{A_Y}$ be the corresponding growth functions (with respect to the standard generating subspaces, spanned by $1$ together with the letters of the underlying alphabet). It follows that $g_{A_X}(n)=\sum_{i=1}^{n} p_X(i)$ and $g_{A_Y}(n)=\sum_{i=1}^{n} p_Y(i)$. As mentioned before, $g_{A_X}(n)\leq np_X(n),g_{A_Y}(n)\leq np_Y(n)$ and $np_X(n)\leq g_{A_X}(cn),np_Y(n)\leq g_{A_Y}(cn)$ for any $c\geq 2$.

Let us turn to estimate $n_k,m_k$.
Let $k\in (d_i,e_i]$. Then $m_k\leq c_1\cdot 5^{(\beta-2)k}$ for some $c_1>0$ which depends only on $\beta$ and $t$.
For $k=2^{d_i}$ we have:
$$
m_{2^{d_i}} \geq (5^{\beta-2})^{2^{d_i}-d_i} \geq  \frac{(5^{2^{d_i}})^{\beta-2}}{2^{3(\beta-2)d_i}}
$$
Let $k\in (d_i,e_i]$ again. Then $n_k\leq c_2\cdot 5^{(\gamma-2)k}$ for some $c_2>0$ which depends only on $\gamma$. Similarly, for $k\in (e_i,d_{i+1}]$, we have $n_k\leq c_3\cdot 5^{(\alpha-2)k}$ for some $c_3>0$ which depends only on $\alpha$ and $t$.
For $k=2^{e_i}$ we have:
$$ n_{2^{e_i}} \geq (5^{\alpha-2})^{2^{e_i}-e_i} \geq \frac{(5^{2^{e_i}})^{\alpha-2}}{2^{3(\alpha-2)e_i}} $$
Similarly, for $k=2^{d_i}$ we have:
$$ n_{2^{d_i}} \geq (5^{\gamma-2})^{2^{d_i}-d_i} \geq (5^{\gamma-2})^{2^{d_i}-d_i} \geq \frac{(5^{2^{d_i}})^{\gamma-2}}{2^{3(\gamma-2)d_i}} $$
We now turn to estimate $\GK(A_X),\GK(A_Y)$.
Let: $$ S=\{r\in \mathbb{N}|\lfloor \log_5 r \rfloor\in \bigcup_{i\geq 1} (d_i,e_i]\},\ \ \ S'=\{r\in \mathbb{N}|\lfloor \log_5 r \rfloor\in \bigcup_{i\geq 1} (e_i,d_{i+1}]\} $$ These are disjoint sets, whose union contains every sufficiently large positive integer.

\textbf{Calculating $\GK(A_Y)$.} Let $r\in S$, say, $k=\lfloor \log_5 r\rfloor \in (d_i,e_i]$. By the above calculations, $$ m_k\leq c_1\cdot 5^{(\beta-2)k}\leq c_1\cdot r^{\beta-2}. $$
Using $(*)$, it follows that $p_Y(r)\leq 10r\cdot m_{k+1}\leq 10c_1\lambda r^{\beta-1}$. Hence $g_{A_Y}(r)\leq ~ 10\lambda c_1\cdot ~ r^\beta$.

For $r\in S'$, say $k=\lfloor \log_5 r\rfloor \in (e_i,d_{i+1}]$.
Then $p_Y(r)\leq 10r\cdot m_{k+1}\leq 10\lambda t r$ and consequently $g_{A_Y}(r)\leq 10\lambda t r^2$. It follows that $\GK(A_Y)\leq \beta$.
Moreover, for $r=5^{2^{d_i}}$ using $(*)$:
$$ p_Y(5^{2^{d_i}}) \geq \frac{2}{25}5^{2^{d_i}}\cdot m_{2^{d_i}} \geq 
\frac{2 \cdot (5^{2^{d_i}})^{\beta-1}}{25\cdot 2^{3(\beta-2)d_i}} $$
so: $$ g_{A_Y}(2\cdot 5^{2^{d_i}}) \geq \frac{2 \cdot (5^{2^{d_i}})^{\beta}}{25\cdot 2^{3(\beta-2)d_i}}. $$
which proves that $\GK(A_Y)=\beta$.

\textbf{Calculating $\GK(A_X)$.} Let $r\in S'$, say, $k=\lfloor \log_5 r\rfloor \in (e_i,d_{i+1}]$. By the above calculations, $$ n_k\leq c_3\cdot 5^{(\alpha-2)k}\leq c_3\cdot r^{\alpha-2}. $$
Using $(*)$, it follows that $p_X(r)\leq 10r\cdot n_{k+1}\leq 10c_3\lambda r^{\alpha-1}$. Hence $g_{A_X}(r)\leq 10c_3\lambda\cdot r^\alpha$.

For $r\in S$, say $k=\lfloor \log_5 r\rfloor \in (d_i,e_i]$.
By the above calculations, $$ n_k\leq c_2\cdot 5^{(\gamma-2)k}\leq c_2\cdot r^{\gamma-2}. $$
Using $(*)$, it follows that $p_X(r)\leq 10r\cdot n_{k+1}\leq 10 c_2 \lambda r^{\gamma-1}$. Hence $g_{A_X}(r)\leq 10 c_2 \lambda \cdot r^\gamma$. It follows that $\GK(A_X)\leq \alpha$.

Moreover, for $r=5^{2^{e_i}}$ using $(*)$:
$$ p_X(5^{2^{e_i}}) \geq \frac{2}{25}5^{2^{e_i}}\cdot n_{2^{e_i}} \geq 
\frac{2 \cdot (5^{2^{e_i}})^{\alpha-1}}{25\cdot 2^{3(\alpha-2)e_i}} $$
so: $$ g_{A_X}(2\cdot 5^{2^{e_i}}) \geq \frac{2 \cdot (5^{2^{e_i}})^{\alpha}}{25\cdot 2^{3(\alpha-2)e_i}} $$
which can be written as $g_{A_X}(r)\geq c_2' \frac{r^\alpha}{(\log_5 r)^{3(\alpha-2)}}$ for some $c_2'>0$ (and infinitely many $r$'s) which proves that $\GK(A_X)=\alpha$.


We need one additional calculation related to $A_X$, to be used in the sequel.
For $r=5^{2^{d_i}}$ using $(*)$:
$$ p_X(5^{2^{d_i}}) \geq \frac{2}{25}5^{2^{d_i}}\cdot n_{2^{d_i}} \geq \frac{2 \cdot (5^{2^{d_i}})^{\gamma-1}}{25\cdot 2^{3(\gamma-2)d_i}} $$
so:
$$ g_{A_X}(2\cdot 5^{2^{d_i}}) \geq \frac{2 \cdot (5^{2^{d_i}})^{\gamma}}{25\cdot 2^{3(\gamma-2)d_i}} $$
We are now ready to calculated the growth of the tensor product $A_X\otimes_F A_Y$.

\textbf{Calculating $\GK(A_X\otimes_F A_Y)$.}
Let $V_X,V_Y$ be the standard generating subspaces of $A_X,A_Y$, respectively. Then $V_X\otimes_F V_Y$ is a generating subspace of $A_X\otimes_F A_Y$. It holds that: $$\left(V_X\otimes_F V_Y\right)^n\subseteq V_X^n\otimes_F V_Y^n\subseteq \left(V_X\otimes_F V_Y\right)^{2n}$$
By the above calculations, for $r\in S$ we have:
$$ g_{A_Y}(r) = O(r^\beta),\ \ \ g_{A_X}(r) = O(r^\gamma) $$
so: $$ \dim_F \left(V_X\otimes_F V_Y\right)^n\leq \dim_F \left(V_X^n\otimes_F V_Y^n\right) = \left( \dim_F V_X^n \right) \cdot \left( \dim_F V_Y^n \right) = O(r^{\beta+\gamma}) $$
And for $r\in S'$:
$$ g_{A_Y}(r) = O(r^2),\ \ \ g_{A_X}(r) = O(r^\alpha) $$
so: $$ \dim_F \left(V_X\otimes_F V_Y\right)^n \leq \dim_F \left(V_X^n\otimes_F V_Y^n\right) = \left( \dim_F V_X^n \right) \cdot \left( \dim_F V_Y^n \right) = O(r^{\alpha+2}) $$
Since $\beta+\gamma \geq \alpha+2$, it follows that $\GK(A_X\otimes_F A_Y)\leq \beta+\gamma$.

For each $i$ let $r_i=5^{2^{d_i}}$ and notice that the above calculations also show that:
$$ \dim_F V_Y^{2r_i} \geq \frac{2r_i^\beta}{25(\log_5 r_i)^{3(\beta-2)}},\ \ \ \dim_F V_X^{2r_i} \geq \frac{2r_i^\gamma}{25(\log_5 r_i)^{3(\gamma-2)}} $$
so:
\begin{eqnarray*}
g_{A_X\otimes_F A_Y}(4r_i) & \geq & \dim_F \left(V_X^{2r_i}\otimes_F V_Y^{2r_i} \right) \\ & = & \left( \dim_F V_X^{2r_i} \right) \cdot \left( \dim_F V_Y^{2r_i} \right) \geq \frac{4r_i^{\beta+\gamma}}{625(\log_5 r_i)^{3(\beta+\gamma-4)}}
\end{eqnarray*}
hence $\GK(A_X\otimes_F A_Y)=\beta+\gamma$.





\begin{proof}[{Proof of Theorem D}]
Given $2\leq \gamma\leq \alpha\leq \beta$ construct $A_X,A_Y$ as above. Since $X,Y$ are Toeplitz sequences of and $A_X,A_Y$ are of \GKdim\ greater than $1$, then by Lemma \ref{words} they are both primitive. By the above calculations:
$$ \GK(A_X)=\alpha,\ \ \GK(A_Y)=\beta,\ \ \GK(A_X\otimes_F A_Y)=\beta+\gamma $$
as required.
\end{proof}

\begin{rem}
Using the growth analysis of convolution algebras associated with minimal subshifts presented in \cite{Nekrashevych}, one can use the above construction to construct examples of \textit{simple} algebras $A,B$ with $\GK(A)=\alpha,\GK(B)=\beta,\GK(A\otimes_FB)=\beta+\gamma$.
\end{rem}

\end{document}